\documentclass[a4paper,11pt]{amsart}
\usepackage{amsmath}
\usepackage{amssymb}
\usepackage{amsthm}
\usepackage[english]{babel}
\usepackage{hyperref}

\newtheorem{theor}{Theorem}
\newtheorem{lemma}[theor]{Lemma}

\newtheorem{cor}[theor]{Corollary}
\newtheorem{conj}[theor]{Conjecture}

\newcommand{\la}{\lambda}
\newcommand{\al}{\alpha}
\newcommand{\be}{\beta}
\newcommand{\Om}{\Omega}

\newcommand{\Md}{\!\mod}

\begin{document}

\title[Divisibility by primes in character tables of symmetric groups]{{\bf On divisibility by primes in columns of character tables of symmetric groups}}

\author{\sc Lucia Morotti}
\address
{Institut f\"{u}r Algebra, Zahlentheorie und Diskrete Mathematik\\ Leibniz Universit\"{a}t Hannover\\ 30167 Hannover\\ Germany} 
\email{morotti@math.uni-hannover.de}


\begin{abstract}
For an arbitrary prime $p$ we prove that the proportion of entries divisible by $p$ in certain columns of the character table of the symmetric group $S_n$ tends to 1 as $n\to\infty$. This is done by finding lower bounds on the number of $k$-cores, for $k$ large enough with respect to $n$.
\end{abstract}

\maketitle

\vspace{12pt}

\section{Introduction}

In \cite{m} Miller formulated the following conjecture about the character table of symmetric groups:
\begin{conj}\label{c2}\label{c1}
Let $p$ be a prime and $E_p(n)$ be the number of entries divisible by $p$ in the character table of $S_n$. Then $E_p(n)/(p(n))^2\to 1$ as $n\to \infty$.
\end{conj}
Here, as in the rest of the paper, $p(n)$ is the number of partitions of $n$.

In \cite{g} Gluck proved that on certain columns of the character table of $S_n$, the proportion of even entries tends to 1. The main results of this paper extend this to a larger set of columns of the character table of $S_n$ and hold for any prime $p$. These results however are not sufficient to prove Conjecture \ref{c2}. In order to state our main results we need the following notation. Let $P(n)$ be the set of all partitions of $n$ and
\[\Om_p(n):=\{\mbox{partitions of }n\mbox{ into parts not divisible by }p\}.\]
Further for any partition $\mu$ of $n$ let $\mu^*\in\Om_p(n)$ be obtained from $\mu=(\mu_1,\mu_2,\ldots)$ by replacing each part $\mu_i=p^{k_i}a_i$ with $p\nmid a_i$ by $p^{k_i}$ parts $a_i$. Moreover, for $\la\in\Om_p(n)$, let
\[K_p(\la):=\{\mu\in P(n)|\mu^*=\la\}.\]

For $\al,\be\in P(n)$ let $\chi^\al$ be the irreducible character of $S_n$ indexed by $\al$ and $\chi^\al_\be$ be the value that $\chi^\al$ takes on the conjugacy class with cycle partition $\be$.

\begin{theor}\label{t1}
Let $p$ be a prime, $n\geq 2$, $c>\frac{\sqrt{3/2}}{\pi}$ and $\la=(a_1^{b_1},\ldots,a_h^{b_h})\in\Om_p(n)$. Assume that for some $1\leq i\leq h$ there exists $s$ with $p^s\leq b_i$ and $a_ip^s\geq c\sqrt{n}\log(n)$. Then for any $\mu\in K_p(\la)$ we have that
\[\frac{|\{\al\in P(n):\chi^\al_\mu\equiv 0\Md p\}|}{p(n)}\geq 1-\frac{c^4d\log(n)}{n^{c\pi/\sqrt{6}-1/2}}\]
for some constant $d$.
\end{theor}

Note that $\frac{\log(n)}{n^{c\pi/\sqrt{6}-1/2}}\to 0$ since $c>\frac{\sqrt{3/2}}{\pi}$.

\begin{cor}\label{c4}
Let $p$ be a prime, $n\geq 2$, $c>\frac{\sqrt{3/2}}{\pi}$ and $\la=(a_1^{b_1},\ldots,a_h^{b_h})\in\Om_p(n)$. If $h\leq \frac{\sqrt{n}}{cp\log(n)}$, then for any $\mu\in K_p(\la)$ we have that
\[\frac{|\{\al\in P(n):\chi^\al_\mu\equiv 0\Md p\}|}{p(n)}\geq 1-\frac{c^4d\log(n)}{n^{c\pi/\sqrt{6}-1/2}}\]
for some constant $d$.
\end{cor}

Corollary \ref{c4} easily follows from Theorem \ref{t1}, since under the assumptions of Corollary \ref{c4} there exists $i$ with $a_ib_i\geq cp\sqrt{n}\log(n)$ and then the assumptions of Theorem \ref{t1} are satisfied.

If $\mu,\nu\in K_p(\la)$ then the two columns of the character table of $S_n$ corresponding to conjugacy classes with cycle partitions $\mu$ and $\nu$ are congruent modulo $p$ (see \cite[Proposition 1]{m}). In particular the numbers of character values divisible by $p$ in the two columns are equal. This explains why Theorem \ref{t1} and Corollary \ref{c4} only have assumptions on $\la=\mu^*$ and not on $\mu$.

\section{Proof of Theorem \ref{t1}}

Given a positive integer $k$ and a partition $\gamma$, we say that $\gamma$ is a $k$-core if $\gamma$ has no hook of length divisible by $k$. For any partition $\beta$ of $n$ and a positive integer $k$, one can define its $k$-core partition $\gamma$ to be the partition obtained from $\beta$ by recursively removing as many $k$-hooks as possible ($\gamma$ does not depend on which maximal sequence of $k$-hooks is removed from $\beta$), thus $|\beta|=|\gamma|+mk$ for a certain non-negative integer $m$ (see for example \cite[Section 3]{o}). 

For any integer $m\geq 0$ let $p_k(m)$ be the number of multipartitions of $m$ into $k$ partitions. For any non-negative integer $m$ and any $k$-core partition $\gamma$ of $n-km$, the number of partitions of $n$ with $k$-core $\gamma$ is always equal to $p_k(m)$ (see for example \cite[Proposition 3.7]{o}).

We start by finding bounds on the number of $k$-core partitions of $n$ when $k$ is large enough. To obtain these bounds we will need bounds on the growth of the number of multipartitions, which will allow us to find lower bounds on $c_k(n)$, the number of $k$-core partitions of $n$. These results will then allow us to prove Theorem \ref{t1} at the end of this section.

\begin{lemma}\label{l1}
Let $k\geq 1$ and $m\geq 1$. Then $p_k(m)\leq (k+1)p_k(m-1)$.
\end{lemma}

\begin{proof}
For $\la=(\la^1,\ldots,\la^k)$ a multipartition of $m-1$ let $h$ be maximal such that $|\la^h|>0$ (set $h=0$ if $m=1$) and let $A(\la)$ be the set of multipartitions of $m$ which can be obtain by adding a node either to $\la^h$ on the last row or the first column or by adding one node to some $\la^i$ with $i>h$. Note that $|A(\la)|\leq k+1$ for each $\la$ and any multipartition of $m$ is contained in $A(\la)$ for some multipartition $\la$ of $m-1$. The result follows.
\end{proof}

\begin{lemma}\label{l2}
For any $1\leq k\leq n$ we have $p(n)-c_k(n)\leq (k+1)p(n-k)$.
\end{lemma}

\begin{proof}
It follows from Lemma \ref{l1} and the classification of partitions with the same $k$-core (see for example \cite[Proposition 3.7]{o}), since
\begin{align*}
p(n)-c_k(n)&=\sum_{m=1}^{\lfloor n/k\rfloor}c_k(n-mk)p_k(m)\\
&\leq (k+1)\sum_{m=1}^{\lfloor n/k\rfloor}c_k(n-k-(m-1)k)p_k(m-1)\\
&=(k+1)p(n-k).
\end{align*}
\end{proof}

\begin{lemma}\label{l3}
Let $n\geq 2$ and $c>\frac{\sqrt{3/2}}{\pi}$. If $k\geq c\sqrt{n}\log(n)$ then
\[\frac{c_k(n)}{p(n)}\geq 1-\frac{c^4d\log(n)}{n^{c\pi/\sqrt{6}-1/2}}\]
for some constant $d$.
\end{lemma}

\begin{proof}
From Lemma \ref{l2} we have that
\[\frac{p(n)-c_k(n)}{p(n)}\leq\frac{(k+1)p(n-k)}{p(n)}.\]

Note that there exist constants $d_1,d_2>0$ such that for any $m\geq 1$
\[\frac{d_1}{m}e^{\pi\sqrt{\frac{2m}{3}}}\leq p(m)\leq \frac{d_2}{m}e^{\pi\sqrt{\frac{2m}{3}}}\]
(see \cite[(1.41)]{hr}). Using the inequalities displayed above, we see that the statement holds for $k=n$, so we may assume that $n-k\geq 1$. Then
\[\frac{(k+1)p(n-k)}{p(n)}\leq \frac{d_2(k+1)n}{d_1(n-k)}e^{-\pi\sqrt{\frac{2n}{3}}(1-\sqrt{1-\frac{k}{n}})}\leq\frac{2d_2kn}{d_1(n-k)}e^{-\frac{\pi k}{\sqrt{6n}}}.\]

If $k\geq 4c\sqrt{n}\log(n)$ then
\[\frac{p(n)-c_k(n)}{p(n)}\leq\frac{2d_2n^2}{d_1}e^{-\frac{4c\pi\log(n)}{\sqrt{6}}}=\frac{2d_2}{d_1n^{4(c\pi/\sqrt{6}-1/2)}}\]
so in this case the lemma holds, since $\frac{c\pi}{\sqrt{6}}-\frac{1}{2}>0$ by assumption on $c$.

If $k=\overline{c}\sqrt{n}\log(n)$ with $c\leq\overline{c}< 4c$ and $k\leq n/2$ then
\[\frac{p(n)-c_k(n)}{p(n)}\leq\frac{4d_2\overline{c}\sqrt{n}\log(n)}{d_1}e^{-\frac{\pi\overline{c}\log(n)}{\sqrt{6}}}\leq \frac{16cd_2\log(n)}{d_1n^{c\pi/\sqrt{6}-1/2}}\]
so that also in this case the lemma holds.

If $k=\overline{c}\sqrt{n}\log(n)$ with $c\leq\overline{c}< 4c$ and $k>n/2$ then $\overline{c}>\sqrt{n}/(2\log(n))$. Since $d_3\sqrt{n}>8(\log(n))^3$ for $d_3$ large enough, there exists a constant $d_3$ such that $n<d_3\overline{c}^3$. It then follows that
\[\frac{p(n)-c_k(n)}{p(n)}\leq\frac{2d_2d_3\overline{c}^4\sqrt{n}\log(n)}{d_1}e^{-\frac{\pi\overline{c}\log(n)}{\sqrt{6}}}\leq \frac{512c^4d_2d_3\log(n)}{d_1n^{c\pi/\sqrt{6}-1/2}}.\]
\end{proof}

We are now ready to prove Theorem \ref{t1}.

\begin{proof}[Proof of Theorem \ref{t1}]
Let $\la=(a_1^{b_1},\ldots,a_h^{b_h})\in\Om_p(n)$ and assume that there exists $i$ and $s$ with $p^s\leq b_i$ and $a_ip^s\geq c\sqrt{n}\log(n)$. For any $1\leq j\leq h$ let $b_j=f_{j,0}p^0+\ldots+f_{j,g_j}p^{g_j}$ be the $p$-adic decomposition of $b_j$ and set $\overline{\delta}_j:=((p^{g_j})^{f_{j,g_j}},\ldots,1^{f_{j,0}})$ and $\overline{\la}:=a_1\overline{\delta}_1\cup\ldots\cup a_h\overline{\delta}_h$
(if $\phi=(\phi_1,\ldots,\phi_r)$ and $\psi=(\psi_1,\ldots,\psi_s)$ are partitions and $t$ is a non-negative integer, then $t\phi=(t\phi_1,\ldots,t\phi_r)$ and $\phi\cup\psi$ is the partition obtained by rearranging the parts of $(\phi_1,\ldots,\phi_r,\psi_1,\ldots,\psi_s)$). Then $\overline{\la}\in K_p(\la)$ and by assumption $\overline{\la}_1\geq c\sqrt{n}\log(n)$.

Note that for any partition $\mu\in K_p(\la)$ we have from \cite[Proposition 1]{m} that
\[\frac{|\{\al\in P(n):\chi^\al_\mu\equiv 0\Md p\}|}{p(n)}=\frac{|\{\al\in P(n):\chi^\al_{\overline{\la}}\equiv 0\Md p\}|}{p(n)}.\]
Since $\overline{\la}_1\geq c\sqrt{n}\log(n)$, the theorem holds for $\overline{\la}$ by Lemma \ref{l3} and the Murnaghan-Nakayama formula. So the statement of the theorem holds also for $\mu$.
\end{proof}

\section*{Acknowledgements}

The author thanks Alexander Miller for bringing this problem to her attention and for some discussion.

The author was supported by the DFG grant MO 3377/1-1.


\begin{thebibliography}{9}

\bibitem{g} D. Gluck, Parity in columns of the character table of $S_n$, {\em Proc. Amer. Math. Soc.} 147 (2019) 1005-1011.

\bibitem{hr}  G. H. Hardy, S. Ramanujan, Asymptotic formulae in combinatory analysis, {\em Proc. Lond. Math. Soc. (2)}, 17 (1918) 75-115.

\bibitem{m} A. R. Miller, On parity and characters of symmetric groups, {\em  J. Combin. Theory Ser. A} 162 (2019) 231-240.

\bibitem{o} J. B. Olsson, {\em Combinatorics and Representations of Finite Groups,} Vorlesungen aus dem Fachbereich Mathematik der Univerit\"{a}t GH Essen, (1994), Heft 20.

\end{thebibliography}
\end{document}